\documentclass{article}
\usepackage[utf8]{inputenc}
\usepackage{amsmath}
\usepackage{amsthm}
\usepackage{amsfonts}
\usepackage{amssymb}
\usepackage{bm}
\usepackage{caption}
\usepackage{xcolor} 
\definecolor{Prune}{RGB}{99,0,60} 
\definecolor{B1}{RGB}{49,62,72} 
\definecolor{C1}{RGB}{124,135,143}
\definecolor{D1}{RGB}{213,218,223}
\definecolor{A2}{RGB}{198,11,70}
\definecolor{B2}{RGB}{237,20,91}
\definecolor{C2}{RGB}{238,52,35}
\definecolor{D2}{RGB}{243,115,32}
\definecolor{A3}{RGB}{124,42,144}
\definecolor{B3}{RGB}{125,106,175}
\definecolor{C3}{RGB}{198,103,29}
\definecolor{D3}{RGB}{254,188,24}
\definecolor{A4}{RGB}{0,78,125}
\definecolor{B4}{RGB}{14,135,201}
\definecolor{C4}{RGB}{0,148,181}
\definecolor{D4}{RGB}{70,195,210}
\definecolor{A5}{RGB}{0,128,122}
\definecolor{B5}{RGB}{64,183,105}
\definecolor{C5}{RGB}{140,198,62}
\definecolor{D5}{RGB}{213,223,61}
\usepackage{mdframed}
\usepackage{multirow} 
\usepackage{multicol} 
\usepackage{tikz}
\usetikzlibrary{cd}
\usepackage{graphicx}
\usepackage[absolute]{textpos} 
\usepackage{colortbl}
\usepackage{array}
\usepackage{geometry}
\usepackage{titlesec}
\usepackage{hyperref}
\hypersetup{ 
    colorlinks=true,
    citecolor=Prune,
    linkcolor=black,
    urlcolor=A5}
\usepackage{amssymb}   
\usepackage{stmaryrd}
\usepackage{amsmath}
\usepackage{amsmath,amscd}
\usepackage{dsfont}
\usepackage{enumitem}

\usepackage{wrapfig}			
\usepackage{standalone}
\usepackage{tikz,pgfplots}
\pgfplotsset{compat=1.15}
\usepackage{mathrsfs}
\usetikzlibrary{arrows}
\usetikzlibrary{decorations.markings}
\usepackage{adjustbox}
\usepackage{caption}
\usepackage{enumitem}
\usepackage{blindtext}
\usepackage{comment}
\usepackage{listings}
\usepackage{algorithm}
\usepackage{placeins}
\usepackage[italicComments=false]{algpseudocodex}

\usepackage{subcaption}
\usepackage{mathtools}
\newcommand{\defeq}{\vcentcolon=}

\usepackage[nottoc]{tocbibind}
\usepackage{tikz-cd}
\newcounter{thmletter}

\renewcommand{\thethmletter}{\arabic{thmletter}}

\newcounter{thmletterfr}

\renewcommand{\thethmletterfr}{\arabic{thmletterfr}}

\counterwithin{figure}{section}

 \newtheorem {definition} {Definition} [section] 
 \newtheorem {theorem}[definition] {Theorem}

  \newtheorem {corollary}[definition]  {Corollary}
    
\newtheorem {prop}[definition] {Proposition}
\newtheorem {lemma}[definition] {Lemma}
 
\newtheorem {remark}[definition] {Remark}

\newtheorem {example}[definition] {Example}

\makeatletter
\let\c@figure\c@definition
\let\c@algorithm\c@definition
\makeatother

\numberwithin{equation}{section}
\newcommand{\theoremcolorbox}[2]{\def\theoremframecommand{\fcolorbox{#1}{#2}}}

\newtheorem{hyp}{Hypothesis}
\theoremcolorbox{black}{cyan}
\newcommand{\encad}[1]{%
\fbox{\begin{minipage}[t]{0.92\linewidth}%
#1\end{minipage}}}

\theoremcolorbox{black}{cyan}

\usepackage{listings}

\lstdefinelanguage{Julia}%
  {morekeywords={abstract,break,case,catch,const,continue,do,else,elseif,%
      end,export,false,for,function,immutable,import,importall,if,in,%
      macro,module,mutable,quote,return,struct,true,try,type,typealias,%
      using,while},%
   sensitive=true,%
   morecomment=[l]{\#},%
   morecomment=[n]{\#=}{=\#},%
   morestring=[s]{"}{"},%
   morestring=[m]{'}{'},%
}[keywords,comments,strings]

\newcommand*{\EnsembleQuotient}[2]%
{\ensuremath{%
    #1/\!\raisebox{-.65ex}{\ensuremath{#2}}}}
    
\newcommand*{\RZ}{\EnsembleQuotient{\mathbb{R}}{\mathbb{Z}}}

\newcommand*{\fonction}[5]{
\begin{align*}
#1 : \left\{\begin{array}{lcl}  #2 &\to      &#3\\
                                #4    &\mapsto  &#5
\end{array} \right. .      
\end{align*} 
}

\newcommand{\N}{\mathbb{N}}

\title{Multi-type Galton-Watson processes in dynamical environments}
\author{Thomas Morand\thanks{Université Paris-Saclay, CNRS, Laboratoire de mathématiques d’Orsay, 91405, Orsay, France\\ Email: thomas.morand@universite-paris-saclay.fr}\\\textit{Paris Saclay University}}
\date{2026}

\begin{document}

\maketitle

\renewcommand{\proofname}{Proof}
\renewcommand\refname{References}
\renewcommand*\contentsname{Table of contents}
\newcommand{\argmin}{\text{argmin}}

\begin{abstract}

We define a model of multi-type Galton-Watson processes in dynamical environments where the environment evolves according to a dynamical system $(\mathbb{X},T)$. Three behaviours are possible: uniformly subcritical, critical, and uniformly supercritical. We provide a criterion to determine the regime of a multi-type Galton-Watson process in dynamical environments. We study the continuity of the probability of extinction $q$ in the uniformly supercritical case.
\end{abstract}

\section*{Introduction}

Galton-Watson processes are population evolution processes that have been well known since the 19th century. A generalisation of the classical model is the consideration of processes involving several types of objects. These different types of individuals have different reproduction laws and can produce offspring of different types. This model is called multi-type Galton-Watson processes. Bartlett~\cite{bartlett1946notes} introduced it in a special case involving two types of individuals. The general formulation and treatment were provided by Kolmogorov and Dmitriev~\cite{MR21264}. We refer notably to \cite{MR163361} for a detailed historical study of multi-type Galton-Watson processes.

In \cite{M1}, we have defined the Galton-Watson processes in dynamical environments. These are Galton-Watson processes in which the law of reproduction across generations evolves according to a dynamical system $(\mathbb{X},T)$.

In this article, we define the multi-type Galton-Watson processes in dynamical environments by extending the preceding model to the multi-type case. Let $d\in\mathbb{N}^*$ be the number of types which is fixed and finite in this article.
We consider the sub-multiplicative norm $\|\cdot\|$ defined by: for all $A=(a_{i,j})_{(i,j)\in\llbracket 1,d\rrbracket}\in\mathcal{M}_d(\mathbb{R})$, \begin{align*}
 \|A\|=\max_{i\in\llbracket 1,d\rrbracket} \sum_{j=1}^d |a_{i,j}|.   
\end{align*}
We denote by $\rho(A)$ the spectral radius of the matrix $A$.

We consider a discrete-time dynamical system defined on a compact space $(\mathbb{X},T)$ together with a continuous function $x\in\mathbb{X}\mapsto \bm{\mu}_x\in\mathcal{P}(\mathbb{N}^d)^d$, called the law of reproduction of the multi-type Galton-Watson process in dynamical environments. The law of reproduction of the multi-type Galton-Watson process in dynamical environments with the initial environment $x\in\mathbb{X}$ at generation $n\in\mathbb{N}$ is $\bm{\mu}_{T^nx}$, i.e., the law of reproduction between generations evolves under the action of $T$. 

The probability of extinction $\bm{q}=(q_1,\ldots,q_d)$ is a function defined on $\mathbb{X}$ such that for all $x\in\mathbb{X}$ and $i\in\llbracket 1,d\rrbracket$, $q_i(x)$ denotes the probability of extinction of the multi-type Galton-Watson process in dynamical environments with the initial environment $x$ and starting with an individual of type $i \in \llbracket1,d\rrbracket$. When $\bm{q}$ is equal to $\bm{1}$, there is almost certain extinction of the process, and we say that $x\in\mathbb{X}$ is a bad environment. We denote by \[
N \defeq \{x \in \mathbb{X} : \bm{q}(x) = \bm{1}\}
\] the set of bad environments. We can choose the initial environment $x$ according to a probability measure $\nu$. In this case, the probability of choosing an initial environment where extinction is almost certain is given by $\nu(N)$. Since $N$ is a $T$-invariant set, its measure with respect to any $T$-invariant and ergodic probability measure on $\mathbb{X}$ is either zero or one.

The results of Athreya and Karlin~\cite[Theorem 8]{MR0298780} and Kaplan~\cite[Theorem 2]{10.1214/aop/1176996659} adapted to this model provide a criterion for determining the measure of $N$ by an ergodic probability measure under some hypotheses. When $\nu$ is a $T$-ergodic probability measure on $\mathbb{X}$, then\begin{align}\label{equ1}
    \nu(N)=0 \text{ if and only if }\pi_\nu>0,
\end{align}
where:\begin{itemize}
    \item $M(x) = (m_{ij}(x))_{1 \leq i,j \leq d}$ is the mean matrix such that for all $i,j\in\llbracket 1,d\rrbracket$, $m_{i,j}$ is the expected number of offspring of type $j$ produced by an individual of type $i$,
    \item the Lyapunov exponent $\pi_\nu$ is defined by\begin{align*} 
        \pi_\nu \coloneqq \lim_{n \to +\infty} \frac{1}{n} \mathbb{E}_\nu[\ln \|M (T^{n-1}\cdot)\dots M(\cdot)\|].
    \end{align*}
\end{itemize}  

We propose a classification of multi-type Galton-Watson processes in dynamical environments (which coincides with the single-type case). The process is: \begin{itemize}
 \item \textbf{uniformly subcritical} if the set $N$ is of measure one according to all ergodic probability measures,
 \item \textbf{critical} if the set $N$ is of measure zero for some ergodic probability measure and of measure one for another,
 \item \textbf{uniformly supercritical} if the set $N$ is of measure zero according to all ergodic probability measures.
\end{itemize}

The main result of this article concerns the regularity of the probability of extinction. We show that the continuity of $x\in\mathbb{X}\mapsto\bm{\mu}_x$, under some hypotheses, is preserved by $x\in\mathbb{X}\mapsto \bm{q}(x)$ (Theorem~\ref{thmcontinuite}). This is a generalization of \cite[Theorem 1.3.5]{M1} to the multi-type case. By definition, the probability of extinction is lower semi-continuous. We show in the uniformly supercritical case that the probability of extinction is also upper semi-continuous. To do that, we assume that the number of offspring of type $j$ produced by an individual of type $i$ under the environment $x$ is bounded from below and above, uniformly in $i$, $j$ and $x$. We then obtain that the ratio between two coefficients of the matrix $M (T^{n-1}x)\dots M(x)$ is bounded by a constant which does not depend on the choice of the coefficients, on $x$ and on $n$. By \cite[Corollary~1.11]{MR1734626} (which is a variation of the semi-uniform Birkhoff ergodic theorem for a sub-additive sequence of continuous functions with a strict inequality condition), we can specify a decreasing sequence of continuous functions greater than $\bm{q}$. Finally, by \cite[Theorems 1 and 2]{MR331477} (which study the composition limits of probability generating functions in the multi-type case), we conclude that this sequence converges to $\bm{q}$.

\subsection*{Outline}

In Section~\ref{sect1}, we present the models of multi-type Galton-Watson processes in the classical case, in varying environments, and in random environments. We recall results on the probability of extinction in these models.

In Section~\ref{sect2}, we define the model of multi-type Galton-Watson in dynamical environments. We will then define the probability generating function $\bm{\varphi}$, the probability of extinction $\bm{q}$, and the set of bad environments $N$.

In Section~\ref{sect3}, we present the main theorems of this article. Theorem~\ref{extmult} provides a criterion to determine the regime of a Galton-Watson process. Theorem~\ref{thmcontinuite} shows, under some hypotheses, that the continuity of $x\in\mathbb{X}\mapsto\bm{\mu}_x$ is indeed preserved by $x\in\mathbb{X}\mapsto \bm{q}(x)$ in the uniformly supercritical case.

Finally, in Section~\ref{sect4}, we provide the proofs of the results presented in this article.

\tableofcontents
\newpage

\section{Other Multi-type Galton-Watson processes}\label{sect1}

In this section, we present the models of multi-type Galton-Watson processes in the classical case, varying environments and random environments. We provide results on the probability of extinction in these models.

\subsection{Classical case}

We now formally introduce the model of multi-type Galton-Watson processes in the classical case. 
Let $d \in \mathbb{N}^*$ be the number of types (which is fixed in this chapter) and, for each $i \in \llbracket1,d\rrbracket$, let $\mu^{(i)} \in \mathcal{P}(\mathbb{N}^d)$ be a probability measure on $\mathbb{N}^d$. The multi-type Galton–Watson process associated with the family of laws of reproduction $\bm{\mu}=(\mu^{(i)})_{1 \leq i \leq d}$ is the sequence of random vectors $(\bm{Z}_n)_{n\in\mathbb{N}}$ in $\mathbb{N}^d$ defined recursively by:
\begin{align}
\left\{\begin{array}{ll}
\bm{Z}_0 &= \bm{z}_0 \in \mathbb{N}^d ,\\
\bm{Z}_{n+1}^{(j)} &= \sum\limits_{i=1}^d \left( \underset{k=1}{\overset{Z_n^{(i)}}{\sum}} Y_{n,k}^{(i,j)} \right) \text{ for all } j \in \llbracket1,d\rrbracket \text{ and } n \in \mathbb{N}, \label{eqmult}
\end{array}\right.
\end{align}
where $(\bm{Y}_{n,k}^{(i)})_{(n,k,i) \in \mathbb{N}^2 \times \llbracket1,d\rrbracket}$ is a family of independent random vectors such that \newline $\bm{Y}_{n,k}^{(i)} = (Y_{n,k}^{(i,1)}, \dots, Y_{n,k}^{(i,d)})$ is distributed according to $\mu^{(i)}$ for all $n,k\in\mathbb{N}$ and $i\in\llbracket1,d\rrbracket$. For all $n,k\in\mathbb{N}$ and $i\in\llbracket1,d\rrbracket$, the vector $\bm{Y}_{n,k}^{(i)} = (Y_{n,k}^{(i,1)}, \dots, Y_{n,k}^{(i,d)})$ represents the number of descendants of the $k$-th individual of type $i$ in the $n$-th generation, where each component $Y_{n,k}^{(i,j)}$ denotes the total number of offspring of type $j$, while the vector $\bm{Z}_n = (Z_n^{(1)}, \dots, Z_n^{(d)})$ represents the composition of the population at the $n$-th generation, where each component $Z_n^{(j)}$ denotes the total number of individuals of type $j$.

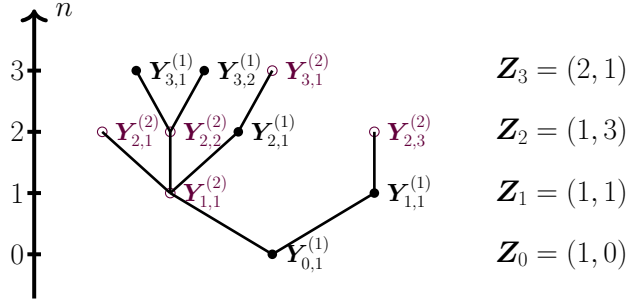
\begin{figure}[!ht]
\centering
\begin{tikzpicture}[scale=0.45, transform shape]
  \draw[->, thick, line width=1.5pt] (11.5, -4.5) -- (11.5, 4); 
  
  \foreach \y/\label in {-3.2/0, -1.4/1, 0.4/2, 2.2/3} {
      \draw[thick,line width=1.5pt] (11.3, \y) -- (11.7, \y); 
      \node at (11, \y) {\fontsize{25}{11}\selectfont\(\label\)}; 
  }

  \node[anchor=west] at (12, 4) {\fontsize{25}{11}\selectfont $n$};

 \coordinate (x0)    at (18.5,-3.2);
 \fill (x0)    circle (4pt);
 \coordinate (x1)    at (15.5,-1.4);
 \draw[Prune] (x1)    circle (4pt);
 \coordinate (x2)    at (21.5,-1.4);
 \fill (x2)    circle (4pt);
 \coordinate (x3)    at (13.5,0.4);
 \draw[Prune] (x3)    circle (4pt);
 \coordinate (x4)    at (15.5,0.4);
 \draw[Prune] (x4)    circle (4pt);
 \coordinate (x5)    at (17.5,0.4);
 \fill (x5)    circle (4pt);
  \coordinate (x6)    at (21.5,0.4);
 \draw[Prune] (x6)    circle (4pt);
 \coordinate (x7)    at (14.5,2.2);
 \fill (x7)    circle (4pt);
 \coordinate (x8)    at (16.5,2.2);
 \fill (x8)    circle (4pt);
  \coordinate (x9)    at (18.5,2.2);
 \draw[Prune] (x9)    circle (4pt);

 \draw[-,line width=1pt] (x0) to (x1);
 \draw[-,line width=1pt] (x0) to (x2);
 \draw[-,line width=1pt] (x1) to (x3);
 \draw[-,line width=1pt] (x1) to (x4);
 \draw[-,line width=1pt] (x1) to (x5);
 \draw[-,line width=1pt] (x2) to (x6);
 \draw[-,line width=1pt] (x4) to (x7);
 \draw[-,line width=1pt] (x4) to (x8);
 \draw[-,line width=1pt] (x5) to (x9);

 \node at (27,-3.2){\fontsize{25}{11}\selectfont$\bm{Z}_0=(1,0)$};
 \node at (27,-1.4){\fontsize{25}{11}\selectfont$\bm{Z}_1=(1,1)$};
 \node at (27,0.4){\fontsize{25}{11}\selectfont$\bm{Z}_2=(1,3)$};
 \node at (27,2.2){\fontsize{25}{11}\selectfont$\bm{Z}_3=(2,1)$};

\node[anchor=west, xshift=7pt,  text=black] at (x0) {\fontsize{20}{11}\selectfont \(\bm{Y}_{0,1}^{(1)}\)};
\node[anchor=west, xshift=7pt,  text=Prune] at (x1) {\fontsize{20}{11}\selectfont \(\bm{Y}_{1,1}^{(2)}\)};
\node[anchor=west, xshift=7pt,  text=black] at (x2) {\fontsize{20}{11}\selectfont \(\bm{Y}_{1,1}^{(1)}\)};
\node[anchor=west, xshift=7pt,  text=Prune] at (x3) {\fontsize{20}{11}\selectfont \(\bm{Y}_{2,1}^{(2)}\)};
\node[anchor=west, xshift=7pt,  text=Prune] at (x4) {\fontsize{20}{11}\selectfont \(\bm{Y}_{2,2}^{(2)}\)};
\node[anchor=west, xshift=7pt,  text=black] at (x5) {\fontsize{20}{11}\selectfont \(\bm{Y}_{2,1}^{(1)}\)};
\node[anchor=west, xshift=7pt,  text=Prune] at (x6) {\fontsize{20}{11}\selectfont \(\bm{Y}_{2,3}^{(2)}\)};
\node[anchor=west, xshift=7pt,  text=black] at (x7) {\fontsize{20}{11}\selectfont \(\bm{Y}_{3,1}^{(1)}\)};
\node[anchor=west, xshift=7pt,  text=black] at (x8) {\fontsize{20}{11}\selectfont \(\bm{Y}_{3,2}^{(1)}\)};
\node[anchor=west, xshift=7pt,  text=Prune] at (x9) {\fontsize{20}{11}\selectfont \(\bm{Y}_{3,1}^{(2)}\)};
\end{tikzpicture}
\caption{Schematic tree representation of a classical multi-type Galton-Watson process. Colors denote the two different types.}\label{Tfig1mult}
\end{figure}

In the multi-type setting, extinction occurs when all types vanish simultaneously. We define the extinction set as:
\begin{align*}
\text{Ext} \defeq \bigcup_{n\geq 0}\{\bm{Z}_n = \bm{0}\},
\end{align*}
where $\bm{0} = (0, \dots, 0) \in \mathbb{N}^d$. The probability of extinction is then represented by a vector $\bm{q} = (q_1, \dots, q_d) \in [0,1]^d$, where for each $i \in \llbracket1,d\rrbracket$, $q_i$ is the probability of extinction starting from a single individual of type $i$:
\begin{align*}
    q_i \defeq \mathbb{P}(\text{Ext} \mid \bm{Z}_0 = \bm{e}_i),
\end{align*}
with $\bm{e}_i$ denoting the $i$-th vector of the canonical basis of $\mathbb{R}^d$. To characterize this vector, we introduce the multi-type probability generating function $\bm{\varphi} = (\varphi_1, \dots, \varphi_d):[0,1]^d \to [0,1]^d$, where for each $i \in \llbracket1,d\rrbracket$, the function $\varphi_i : [0,1]^d \to [0,1]$ is associated with the law $\mu^{(i)} \in \mathcal{P}(\mathbb{N}^d)$ and defined by:
\begin{align*}
\varphi_i(\bm{s}) = \sum_{k \in \mathbb{N}^d} \mu^{(i)}(k) \left( \prod_{j=1}^d s_j^{k_j} \right)  \text{ for all } \bm{s} = (s_1, \dots, s_d) \in [0,1]^d.
\end{align*}

\begin{definition}
Let $M = (m_{ij})_{1 \leq i,j \leq d}$ be the square matrix of size $d \times d$ where each entry $m_{ij}$ represents the expected number of individuals of type $j$ produced by a single parent of type $i$:
\begin{align*}
m_{ij} \defeq \mathbb{E}[Z_1^{(j)} \mid \bm{Z}_0 = \bm{e}_i] = \frac{\partial \varphi_i}{\partial s_j}(\bm{1}).
\end{align*}
The matrix $M$ is called the \textbf{mean matrix} of the process. We denote by $\rho$ its \textbf{spectral radius}, which is the largest positive eigenvalue of $M$ according to the Perron-Frobenius theorem.
\end{definition}

\begin{definition}
A multi-type Galton-Watson process $(\bm{Z}_n)_{n\in\mathbb{N}}$ is said to be:\begin{itemize}
    \item \textbf{positively regular} if the matrix $M^N$ is positive for some $N\in\mathbb{N}$,
    \item \textbf{singular} if each individual has exactly one offspring, regardless of type.
\end{itemize} 
\end{definition}

\begin{theorem}\cite{MR26280,MR24090}
Let $(\bm{Z}_n)_{n\in\mathbb{N}}$ be a positively regular and not singular multi-type Galton-Watson process. The probability of extinction vector $\bm{q}$ satisfies the following properties:
\begin{itemize}
    \item if $\rho \leq 1$, then $\bm{q} = \bm{1}$ (almost sure extinction),
    \item if $\rho > 1$, then $\bm{q} < \bm{1}$ (positive probability of survival).
\end{itemize}
Furthermore, $\bm{q}$ is a solution of the fixed-point equation:
\begin{align*}
\bm{q} = \bm{\varphi}(\bm{q}).
\end{align*}
\end{theorem}

As in the single-type case, Kesten and Stigum~\cite{10.1214/aoms/1177699266} give the limit in law of the process by normalizing by the spectral radius.

\subsection{Varying and random environments}

We will not go into detail about the definition of multi-type Galton-Watson processes in varying and random environments. We must preserve the recurrence relation \eqref{eqmult} while changing the law of reproduction:\begin{itemize}
    \item in a deterministic way for the varying environments,
    \item randomly (and independently of the size of the population) for the random environments.
\end{itemize}

Multi-type Galton-Watson processes in varying environments have been studied notably in~\cite{MR1258179,MR1459270,MR1708214,MR2392694}. 

In the following, we will focus on the multi-type Galton-Watson processes in random environments. We further assume that, almost surely:
\begin{align}
    0 \leq \mu^{(i)}_{\xi_0}(\bm{0}) + \sum_{j=1}^d \mu^{(i)}_{\xi_0}(\bm{e}_j) < 1, \quad &\text{for } 1 \leq i \leq d,\label{eqmult1}\\
  \text{and}  \qquad m_{ij}(\xi_0) =  \frac{\partial \varphi_{i,\xi_0}(\bm{1})}{\partial s_j} < +\infty, \quad &\text{for all } 1 \leq i, j \leq d.\nonumber
\end{align}

In the random case, the key object for studying the extinction problem is :
\begin{equation*}
\pi \coloneqq \lim_{n \to +\infty} \frac{1}{n} \mathbb{E}[\ln \|M (\xi_{n-1})\dots M(\xi_{0})\|],
\end{equation*}
where $\xi=(\xi_n)_{n\in\N}$ is the environment and $M$ the mean matrix. The Lyapunov exponent $\pi$ exist and is finite by \cite{MR0121828} (under the hypotheses of Theorem~\ref{multext}). It represents the average growth rate of the population size. The following result gives the probability of extinction of the process:

\begin{theorem}[\textnormal{\cite[Theorem 8]{MR0298780} and \cite[Theorem 2]{10.1214/aop/1176996659}}]\label{multext}
Let $(\bm{Z}_n)_{n\in\mathbb{N}}$ be a multi-type Galton-Watson process in a stationary and ergodic random environment. We assume that there exist constants $0 < C \leq D < +\infty$ such that, almost surely:
    \begin{equation}
    C \leq \min_{i,j} m_{ij}(\xi_0) \leq \max_{i,j} m_{ij}(\xi_0) \leq D \quad \text{and}\quad\max_{i,j,k} \frac{\partial^2 \varphi_{k,\xi_0}}{\partial s_i \partial s_j}(\bm{1}) \leq D.\label{eqmult2}
    \end{equation}
Under these conditions, the probability of extinction vector $\bm{q}$ satisfies:
\begin{itemize}
    \item If $\pi \leq 0$, then $\bm{q} = \bm{1}$ (almost sure extinction).
    \item If $\pi > 0$, then $\bm{q} < \bm{1}$ (positive probability of survival).
\end{itemize}
\end{theorem}

\begin{remark}
Equation~\eqref{eqmult2} implies that the process is positively regular (with $N=1$), while Equation~ \eqref{eqmult1} implies that the process is not singular.
\end{remark}

In the random case, the speed of extinction and the limit in law are studied notably in \cite{MR3846842,MR4461477,MR4564425,MR4838435}.

\section{Multi-type Galton-Watson processes in dynamical environments}\label{sect2}

In this section, we define the model of multi-type Galton-Watson in dynamical environments. We will then define the probability generating function $\bm{\varphi}$, the probability of extinction $\bm{q}$, and the set of bad environments $\bm{N}$.

\subsection{The model}

In this subsection, we present the model of multi-type Galton-Watson processes in dynamical environments.

We consider $(\mathbb{X},\mathcal{B}(\mathbb{X}),T)$ a discrete-time topological dynamical system, where: \begin{itemize}
\item $(\mathbb{X},d)$ is a compact metric space equipped with its Borel algebra $\mathcal{B}(\mathbb{X})$,
\item $T:\mathbb{X}\to\mathbb{X}$ is a continuous map.
\end{itemize}
Additionally, let:
\begin{align*}
\bm{\mu}:\left\{
\begin{array}{lcl}
\mathbb{X}&\to&\mathcal{P}(\mathbb{N}^d)^d\\
x&\mapsto&\bm{\mu}_x=(\mu_x^{(1)},\ldots,\mu_x^{(d)})
\end{array}
\right. ,
\end{align*}
and assume Hypotheses~(H\ref{hmult}).

\encad{
\begin{hyp}[H\ref{hmult}]\label{hmult}~
\begin{enumerate}[label=\alph*)]
\item $x\in\mathbb{X}\mapsto \bm{\mu}_x$ is continuous.\label{1a}
\item $0 \leq \mu^{(i)}_{x}(\bm{0}) + \sum\limits_{j=1}^d \mu^{(i)}_{x}(\bm{e}_j) < 1$ for all $x\in\mathbb{X}$ and $i\in\llbracket 1,d\rrbracket$.\label{1b}
\end{enumerate}
\end{hyp}}\medskip

For each initial environment $x \in \mathbb{X}$, let $(\bm{Z}_n(x))_{n\in\mathbb{N}}$ be the sequence of random vectors in $\mathbb{N}^d$ defined recursively by Equation~\ref{eqmult}, where $(\bm{Y}_{n,k}^{(i)})_{(n,k,i) \in \mathbb{N}^2 \times \llbracket1,d\rrbracket}$ is a family of independent random vectors such that, for all $n, k$ and $i$, $\bm{Y}_{n,k}^{(i)}$ is distributed according to $\mu_{T^n x}^{(i)}$.

The sequence $(\bm{Z}_n(x))_{n\in\mathbb{N}}$ is called \textbf{the multi-type Galton–Watson process in dynamical environments} associated with the discrete-time dynamical system $(\mathbb{X},\mathcal{B}(\mathbb{X}),T)$, the law of reproduction $\bm{\mu}$, and the initial environment $x\in\mathbb{X}$.

The following example will be used several times.

\begin{example}\label{ex_multitype} 
Let $d=2$. We consider the following multi-type process:
\begin{itemize}
\item $\mathbb{X}\defeq\RZ$,
\item $ T\defeq \left\{
    \begin{array}{lcl}\RZ&\to&\RZ\\
    x &\mapsto &2x \text{ modulo } 1\end{array}\right.,$
    \item For all $\lambda\in\mathbb{R}$ and $x\in\RZ$, the reproduction law $\bm{\mu}_{\lambda, x} = \left(\mu_{\lambda, x}^{(1)}, \mu_{\lambda, x}^{(2)}\right)$ is defined as follows:
    \begin{itemize}
        \item $\mu_x^{(1)} \sim \left( \text{Pois}\left(e^{\lambda - \cos(2\pi x)}\right), \text{Pois}\left(e^{\lambda - \sin(2\pi x)}\right) \right)$,
        \item $\mu_x^{(2)} \sim \left( \text{Pois}\left(e^{\lambda + \sin(2\pi x)}\right), \text{Pois}\left(e^{\lambda - \cos(2\pi x)}\right) \right)$,
    \end{itemize}
    where the components of each couple are independent.
\end{itemize}
For each $\lambda \in \mathbb{R}$, this defines a multi-type Galton-Watson process in dynamical environments satisfying Hypotheses~(H\ref{hmult}).
\end{example}

\subsection{Extinction and generating functions}

 We will start with probability generating functions.

\begin{definition}\label{mula2def2}
The probability generating function of the law of reproduction $\bm{\mu}$ is defined on $\mathbb{X}\times [0,1]^d$ by the vector-valued map $\bm{\varphi}(x, \bm{s}) = (\varphi_1(x, \bm{s}), \dots, \varphi_d(x, \bm{s}))$, where for each $i \in \llbracket1,d\rrbracket$:
\begin{align*}
\varphi_i(x, \bm{s}) \defeq \sum_{\bm{k} \in \mathbb{N}^d} \mu_x^{(i)}(\bm{k}) \bm{s}^{\bm{k}} = \mathbb{E}[\bm{s}^{Y(x,i)}],
\end{align*}
where $Y(x,i)$ follows the law $\mu_{x}^{(i)}$.

For all $n \in \mathbb{N}$, the probability generating function of the law of $\bm{Z}_n(x)$ is defined on $\mathbb{X}\times [0,1]^d$ by:
\begin{align*}
\bm{\varphi}^{(n)}(x, \bm{s}) \defeq \big(\mathbb{E}[\bm{s}^{\bm{Z}_n(x)}\mid \bm{Z}_0 = \bm{e}_i]\big)_{i\in\llbracket 1,d\rrbracket}.
\end{align*}
\end{definition}

\begin{example}[Example~\ref{ex_multitype}]
We can explicitly compute the probability generating function. Let $\lambda \in \mathbb{R}$. Then, for all $x \in \RZ$ and $\bm{s} = (s_1, s_2) \in [0,1]^2$, the components of $\bm{\varphi}_\lambda(x, \bm{s})$ are given by:
\begin{align*}
&\varphi_{1,\lambda}(x, \bm{s}) = \exp\left( e^{\lambda-\cos(2\pi x)}(s_1-1) + e^{\lambda-\sin(2\pi x)}(s_2-1) \right), \\
&\varphi_{2,\lambda}(x, \bm{s}) = \exp\left( e^{\lambda+\sin(2\pi x)}(s_1-1) + e^{\lambda-\cos(2\pi x)}(s_2-1) \right). 
\end{align*}
\end{example}

We find a recurrence relation on the population at different generations.

\begin{prop}\label{mulprop5}
For all $x \in \mathbb{X}$, $n,k\in \mathbb{N}$, and $\bm{s} \in [0,1]^d$:
\begin{align}\label{muleq5}
\bm{\varphi}^{(n+k)}(x, \bm{s}) = \bm{\varphi}^{(k)}(x, \bm{\varphi}^{(n)}(T^k x, \bm{s})).
\end{align}
\end{prop}

\begin{figure}[!ht]
\center
\includegraphics[scale=0.5]{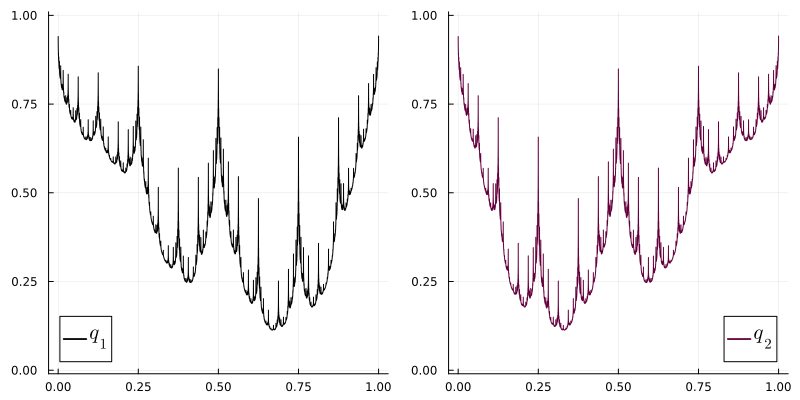}
\caption{Plot of the probability of extinction $\bm{q}_\lambda=(q_{1,\lambda},q_{2,\lambda})$ in the case of Example~\ref{ex_multitype} with $\lambda=-0.4$.}
\end{figure}

We define the probability of extinction, which depends on the initial environment $x\in\mathbb{X}$ and on the initial type $i \in \llbracket1,d\rrbracket$.

\begin{definition}\label{mula2def_q}
We define the probability of extinction vector $\bm{q}:\mathbb{X}\to[0,1]^d$ by setting, for all $x\in\mathbb{X}$ and each component $i \in \llbracket1,d\rrbracket$:
\begin{align*}
q_i(x)\defeq\mathbb{P}\Bigg(\bigcup_{n\geq 0}\{\bm{Z}_n(x)=\bm{0}\} \mid \bm{Z}_0 = \bm{e}_i \Bigg).
\end{align*}
\end{definition}

Probability generating functions can be used to characterize the probability of extinction.

\begin{prop}\label{mulpropq}
For all $x\in\mathbb{X}$:
\begin{align*}
\bm{q}(x)=\underset{n\to\infty}{\lim}\nearrow\big(\mathbb{P}(\bm{Z}_n(x)=\bm{0}\mid \bm{Z}_0 = \bm{e}_i)\big)_{i\in\llbracket 1,d\rrbracket}=\underset{n\to\infty}{\lim}\nearrow \bm{\varphi}^{(n)}(x,\bm{0}).
\end{align*}      
\end{prop}

\begin{prop}\label{mulprop1}
Assume (H\ref{hmult}). Then, for all $x\in\mathbb{X}$:
\begin{align}\label{mule}
\bm{q}(x)=\bm{\varphi}(x, \bm{q}(Tx)).
\end{align}
\end{prop}

\begin{example}\label{ex_multitype2} 
We modify Example~\ref{ex_multitype} by considering a different transformation for the reproduction of each type of individual.

\begin{figure}[!ht]
\center
\includegraphics[scale=0.7]{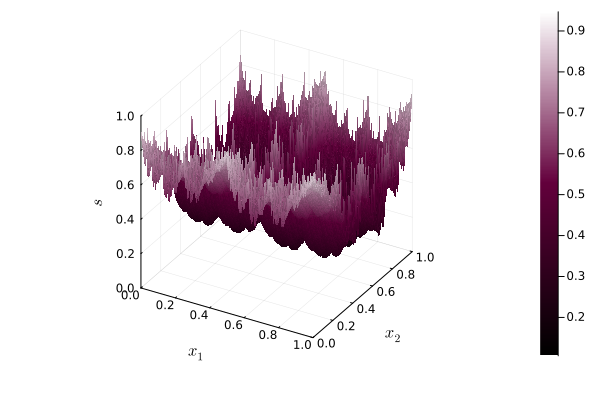}
\caption{Plot of the probability of extinction $q_{1,\lambda}$ (with the initial population $\bm{e}_1$) in the case of Example~\ref{ex_multitype2} with $\lambda=-0.4$.}
\end{figure}

\begin{itemize}
    \item $\mathbb{X} \defeq \RZ^2$,
    \item $T \defeq \left\{
    \begin{array}{lcl}
        \RZ^2 & \to & \RZ^2 \\
        (x_1, x_2) & \mapsto & (2x_1 \text{ modulo } 1, 2x_2 \text{ modulo } 1)
    \end{array}\right.,$
    \item For all $\lambda \in \mathbb{R}$ and $x = (x_1, x_2) \in \RZ^2$, the reproduction law $\bm{\mu}_{\lambda, x} = (\mu_{\lambda, x}^{(1)}, \mu_{\lambda, x}^{(2)})$ is defined as follows:
    \begin{itemize}
        \item $\mu_{\lambda, x}^{(1)} \sim \left( \text{Pois}\left(e^{\lambda - \cos(2\pi x_1)}\right), \text{Pois}\left(e^{\lambda - \sin(2\pi x_1)}\right) \right)$,
        \item $\mu_{\lambda, x}^{(2)} \sim \left( \text{Pois}\left(e^{\lambda + \sin(2\pi x_2)}\right), \text{Pois}\left(e^{\lambda - \cos(2\pi x_2)}\right) \right)$,
    \end{itemize}
    where the components of each couple are independent.
\end{itemize}

Once again, for each $\lambda \in \mathbb{R}$, this defines a multi-type Galton-Watson process in dynamical environments satisfying Hypotheses~(H\ref{hmult}). Contrary to Example~\ref{ex_multitype}, if $\RZ^2$ is equipped with a product measure, then the two populations have independent environments.

\end{example}

The expected number of offspring remains the key parameter in determining the probability of extinction.

\begin{definition}
For all $x\in\mathbb{X}$, let $M(x) = (m_{ij}(x))_{1 \leq i,j \leq d}$ be the mean matrix defined by:
\begin{align*}
M(x)\defeq \left(\frac{\partial \varphi_j}{\partial s_i}(x, \bm{1})\right)_{i,j\in\llbracket 1,d\rrbracket}.
\end{align*}
For all $x\in\mathbb{X}$, $m_{ij}(x) = \frac{\partial \varphi_j}{\partial s_i}(x, \bm{1})$ is the expected number of offspring of type $i$ produced by an individual of type $j$.
\end{definition}

\begin{example}[Example~\ref{ex_multitype}]
Let $\lambda\in\mathbb{R}$. The mean matrix $M_\lambda(x)$ is given by:
\begin{equation*}
M_\lambda(x) = \begin{pmatrix} 
e^{\lambda - \cos(2\pi x)} & e^{\lambda - \sin(2\pi x)} \\
e^{\lambda + \sin(2\pi x)} & e^{\lambda - \cos(2\pi x)}
\end{pmatrix}.
\end{equation*}
\end{example}

Finally, we will examine the environments that lead to almost certain extinction.

\begin{definition}
The set of bad environments is defined by:
\begin{align*}
    N\defeq\{x\in\mathbb{X}:\bm{q}(x)=\bm{1}\}.
\end{align*}
\end{definition}

\section{Results}\label{sect3}

In this section, we present the main theorems of this article. Theorem~\ref{extmult} provides a criterion to determine the regime of a Galton-Watson process. Theorem~\ref{thmcontinuite} shows, under some hypotheses, that the continuity of $x\in\mathbb{X}\mapsto\mu_x$ is indeed preserved by $x\in\mathbb{X}\mapsto q(x)$ in the uniformly supercritical case.

\subsection{Measure of the set of bad environments}

We can classify the multi-type Galton-Watson processes in dynamical environments according to the following definition (which coincides with the single-type case \cite{M1}). 
\begin{definition}\label{defclass}
A Galton-Watson process in dynamical environments is said to be: 
\begin{itemize} 
\item \textbf{uniformly subcritical} if the set $N$ has measure one for all ergodic probability measures, 
\item \textbf{critical} if the set $N$ has measure zero for at least one ergodic probability measure and measure one for another, 
\item \textbf{uniformly supercritical} if the set $N$ has measure zero for all ergodic probability measures. 
\end{itemize}
\end{definition} 
 For multi-type Galton–Watson processes in dynamical environments, if we fix an ergodic measure $\nu \in \mathcal{E}_T(\mathbb{X})$ and choose the initial environment $x \in \mathbb{X}$ according to the law $\nu$, then the process $(\bm{\mu}_{T^n x})_{n\in\mathbb{N}}$ is stationary and ergodic, so we can apply the results of Athreya and Karlin~\cite[Theorem 8]{MR0298780} and Kaplan~\cite[Theorem 2]{10.1214/aop/1176996659}.

\encad{
\begin{hyp}[H\ref{hmult2}]\label{hmult2}~
\begin{enumerate}[label=\alph*)]
\item $x\in\mathbb{X}\mapsto \bm{\mu}_x$ is continuous.
\item $0 \leq \mu^{(i)}_{x}(\bm{0}) + \sum\limits_{j=1}^d \mu^{(i)}_{x}(\bm{e}_j) < 1$ for all $x\in\mathbb{X}$.\label{Hb}
\item There exist constants $0 < C \leq D < +\infty$ such that:
\begin{equation*}
    C \leq \min_{i,j} m_{ij} \leq \max_{i,j} m_{ij} \leq D \quad \text{and}\quad\max_{i,j,k} \frac{\partial^2 \varphi_{k}}{\partial s_i \partial s_j}(\cdot,\bm{1}) \leq D .
    \end{equation*}\label{Hc}
\end{enumerate}
\end{hyp}}\medskip

Hypothesis~(H\ref{hmult2}) implies Hypothesis~(H\ref{hmult}).

\begin{theorem}\label{extmult}
    Assume (H\ref{hmult2}). Let $\nu\in\mathcal{E}_T(\mathbb{X})$. Then,
    \begin{align*}
        \nu(N)=0 \qquad\text{if and only if}\qquad \pi_\nu>0,
    \end{align*} 
    where
    \begin{align*}
        \pi_\nu \coloneqq \lim_{n \to \infty} \frac{1}{n} \mathbb{E}_\nu[\ln \|M (T^{n-1}\cdot)\dots M(\cdot)\|].
    \end{align*}
    (The Lyapunov exponent $\pi_\nu$ exist and is finite by \cite{MR0121828}).
\end{theorem}

\begin{remark}In the context of Theorem~\ref{extmult}, for $\nu$-almost all $x\in\mathbb{X}$, \begin{align*}
    \pi_\nu = \lim\limits_{n \to \infty} \frac{1}{n} \ln \|M (T^{n-1}x)\dots M(x)\|.
\end{align*}
\end{remark}

\begin{example}In the case of Examples~\ref{ex_multitype} or~\ref{ex_multitype2}, Hypotheses~(H\ref{hmult2}) is satisfied for all $\lambda \in \mathbb{R}$.
\end{example}

\subsection{Continuity in the supercritical case}

In the uniformly supercritical case, the following results give the continuity of the application $x\in\mathbb{X}\mapsto q(x)$ under some conditions. 

\begin{theorem}\label{thmcontinuite}
    Assume (H\ref{hmult2}) and that $\pi_\nu>0$ for all $\nu\in\mathcal{E}_T(\mathbb{X})$. Then $x\in\mathbb{X}\mapsto \bm{q}(x)$ is continuous.
\end{theorem}

\begin{corollary}\label{multcor6}
    Assume (H\ref{hmult2}). For each $a\in[0,1)^d$, the sequence of functions $(\bm{\varphi}^{(n)}(\cdot,a))_{n\in\mathbb{N}}$ converges uniformly to $\bm{q}$.
\end{corollary}

\section{Proofs}\label{sect4}

In this section, we provide the proofs of the results presented in this article.

\subsection{Elementary results}

In this subsection, we prove some elementary results on multi-type Galton-Watson processes in dynamical environments. We will start with some results on probability generating functions.

\begin{proof}[Proof of Proposition~\ref{mulprop5}]
We first show by induction that for all $x \in \mathbb{X}$, $n \in \mathbb{N}$, and $\bm{s} \in [0,1]^d$:
\begin{align}\label{muleq_induc}
    \bm{\varphi}^{(n+1)}(x, \bm{s}) = \bm{\varphi}^{(n)}(x, \bm{\varphi}(T^n x, \bm{s})).
\end{align}
Let $x \in \mathbb{X}$, $n \in \mathbb{N}$, and $\bm{s} \in [0,1]^d$. Using that $\{\bm{Y}_{n,k}^{(i)}, k \in \mathbb{N}\}$ are i.i.d.\ random vectors and are independent of $\bm{Z}_n(x)$,
\begin{align*}
    \bm{\varphi}^{(n+1)}(x, \bm{s}) 
    &= \mathbb{E}\left[\bm{s}^{Z_{n+1}(x)}\right] \\
    &= \mathbb{E}\left[\bm{s}^{\sum\limits_{i=1}^d \sum\limits_{k=1}^{Z_{n,i}(x)} \bm{Y}_{n,k}^{(i)}}\right] \\
    &= \mathbb{E}\left[ \mathbb{E} \left[ \prod_{i=1}^d \prod_{k=1}^{Z_{n,i}(x)} \bm{s}^{\bm{Y}_{n,k}^{(i)}} \Bigg| \bm{Z}_n(x) \right] \right] \\
    &= \mathbb{E}\left[ \prod_{i=1}^d \mathbb{E} \left[ \bm{s}^{\bm{Y}_{n,1}^{(i)}} \right]^{Z_{n,i}(x)} \right] \\
    &= \mathbb{E}\left[ \prod_{i=1}^d \left( \varphi_i(T^n x, \bm{s}) \right)^{Z_{n,i}(x)} \right] \\
    &= \mathbb{E}\left[ \bm{\varphi}(T^n x, \bm{s})^{\bm{Z}_n(x)} \right] \\
    &= \bm{\varphi}^{(n)}(x, \bm{\varphi}(T^n x, \bm{s})).
\end{align*} The general case follows by induction on $k\in\mathbb{N}$.
\end{proof}

\begin{corollary}\label{mulcor5}
    Under Hypothesis~(H\ref{hmult}), for all $n\in\mathbb{N}$, $\bm{\varphi}^{(n)}$ is continuous on $\mathbb{X}\times[0,1]^d$.
\end{corollary}

\begin{proof}
Let $x,y\in\mathbb{X}$ and $\bm{s}, \bm{t}\in[0,1]^d$. For any component $i \in \llbracket 1, d \rrbracket$, we have:
\begin{align*}
    |\varphi_i(x,\bm{s})-\varphi_i(y,\bm{t})| &\leq|\varphi_i(x,\bm{s})-\varphi_i(x,\bm{t})|+|\varphi_i(x,\bm{t})-\varphi_i(y,\bm{t})|\\
    &\leq|\varphi_i(x,\bm{s})-\varphi_i(x,\bm{t})|+ \sum_{\bm{k} \in \mathbb{N}^d}|\mu_x^{(i)}(\bm{k})-\mu_y^{(i)}(\bm{k})|\bm{t}^{\bm{k}}\\
    &\leq|\varphi_i(x,\bm{s})-\varphi_i(x,\bm{t})|+ \sum_{\bm{k} \in \mathbb{N}^d}|\mu_x^{(i)}(\bm{k})-\mu_y^{(i)}(\bm{k})| ~\text{since}~ \bm{t} \in [0,1]^d\\
    &=|\varphi_i(x,\bm{s})-\varphi_i(x,\bm{t})|+ \|\mu_x^{(i)}-\mu_y^{(i)}\|_{1}.
\end{align*}
The continuity of $\bm{\varphi}$ follows from the continuity of the maps $x\in\mathbb{X}\mapsto\mu_x^{(i)}$ and the continuity of the maps $\bm{s}\in[0,1]^d\mapsto\varphi_i(x,\bm{s})$ for all $x\in\mathbb{X}$. 

For $n\in\mathbb{N}$, the continuity of $\bm{\varphi}^{(n)}$ follows by induction, using Proposition~\ref{mulprop5}.
\end{proof}

\begin{prop}\label{contM}
Under (H\ref{hmult2}), for all $n\in\mathbb{N}^*$, $M^{(n)}\defeq M(T^{n-1}\cdot)\ldots M(\cdot)$ is continuous on $\mathbb{X}$.
\end{prop}

\begin{proof}
Let $x,y\in\mathbb{X}$. 
\begin{align*}
    \|M(x)-M(y)\|&\leq \max_{i\in\llbracket 1,d\rrbracket} \sum_{j=1}^d |\frac{\partial \varphi_j}{\partial s_i}(x, \bm{1})-\frac{\partial \varphi_j}{\partial s_i}(y, \bm{1})|\\
    &\leq \max_{i\in\llbracket 1,d\rrbracket} \sum_{j=1}^d \sum_{k=0}^{+\infty}k|\mu^{(j)}_x(k\bm{e}_i)-\mu^{(j)}_y(k\bm{e}_i)|.
\end{align*}
The continuity of $M$ follows by the continuity of $x\mapsto \mu_x$ and because $C \leq \min_{i,j} m_{ij} \leq \max_{i,j} m_{ij} \leq D$. By induction on $n\in\mathbb{N}^*$, $M^{(n)}$ is continuous for all $n\in\mathbb{N}^*$.
\end{proof}

The followings result concerns the probability of extinction.

\begin{proof}[Proof of Proposition~\ref{mulpropq}]
Let $x \in \mathbb{X}$. For each type $i \in \llbracket 1, d \rrbracket$, we denote by $\mathbb{P}_i$ the probability measure associated with the process starting from a single individual of type $i$, i.e., $\bm{Z}_0(x) = \bm{e}_i$. 
The sequence of events $(\{\bm{Z}_n(x) = \bm{0}\})_{n \in \mathbb{N}}$ is non-decreasing for the inclusion since $\bm{Z}_n(x) = \bm{0}$ implies $\bm{Z}_{n+1}(x) = \bm{0}$. Thus, by definition of $\bm{q}$, 
\begin{align*}
    q_i(x) = \mathbb{P}_i \left( \bigcup_{n \geq 0} \{ \bm{Z}_n(x) = \bm{0} \} \right) = \lim_{n \to \infty} \nearrow \mathbb{P}_i \left( \bm{Z}_n(x) = \bm{0} \right).
\end{align*}
Using the property of the probability generating function, we have\begin{align*}
    \mathbb{P}_i (\bm{Z}_n(x) = \bm{0}) = \varphi^{(n)}_i(x, \bm{0}).
\end{align*} Thus, we obtain:
\begin{align*}
    \bm{q}(x) &= \lim_{n \to \infty} \nearrow \left( \mathbb{P}_i (\bm{Z}_n(x) = \bm{0}) \right)_{i \in \llbracket 1,d \rrbracket} = \lim_{n \to \infty} \nearrow \bm{\varphi}^{(n)}(x, \bm{0}).\qedhere
\end{align*}
\end{proof}

\begin{proof}[Proof of Proposition~\ref{mulprop1}]
Let $x \in \mathbb{X}$. According to Proposition~\ref{mulpropq}:
\begin{align*}
    \bm{q}(x) = \lim_{n \to \infty} \bm{\varphi}^{(n+1)}(x, \bm{0}).
\end{align*}
By applying the recurrence relation established in Proposition~\ref{mulprop5}, we have for all $n \in \mathbb{N}$:
\begin{align*}
    \bm{\varphi}^{(n+1)}(x, \bm{0}) = \bm{\varphi}(x, \bm{\varphi}^{(n)}(Tx, \bm{0})).
\end{align*}
The probability generating function $\bm{\varphi}(x, \cdot)$ is continuous on $[0,1]^d$ for all $x \in \mathbb{X}$ (Corollary~\ref{mulcor5}). Therefore, by taking the limit as $n \to \infty$, we obtain:
\begin{align*}
    \bm{q}(x) &= \lim_{n \to \infty} \bm{\varphi}(x, \bm{\varphi}^{(n)}(Tx, \bm{0})) \\
    &= \bm{\varphi}(x, \lim_{n \to \infty} \bm{\varphi}^{(n)}(Tx, \bm{0})) \\
    &= \bm{\varphi}(x, \bm{q}(Tx)).\qedhere
\end{align*}
\end{proof}

Finally, the last result concerns the set of bad environments.

\begin{prop}\label{mulcor2}
Assume Hypothesis~(H\ref{hmult}). Let $N = \{x \in \mathbb{X} : \bm{q}(x) = \bm{1}\}$. Then, $N = T^{-1}N$. In particular, if $\nu \in \mathcal{E}_T(\mathbb{X})$, then $\nu(N) \in \{0,1\}$.
\end{prop}

\begin{proof}
Let $\nu \in \mathcal{E}_T(\mathbb{X})$. By definition of the set $N$:
\begin{align*} 
N &= \{x \in \mathbb{X} : \bm{q}(x) = \bm{1}\} \\
  &= \{x \in \mathbb{X} : \bm{\varphi}(x, \bm{q}(Tx)) = \bm{1}\} \text{ by Proposition~\ref{mulprop1}},\\
  &= \{x \in \mathbb{X} : \bm{\varphi}(x, \bm{q}(Tx)) = \bm{1}\}\text{ because } \mu^{(i)}_{x}(\bm{0}) < 1 \text{ for all } x\in\mathbb{X}\text{ and }i\in\llbracket 1,d\rrbracket,\\
  &= T^{-1}N.\qedhere
\end{align*}
\end{proof}

\subsection{Continuity of the probability of extinction in the supercritical case}

\paragraph{The mean matrix set}~\newline

\begin{definition}
Let $C>0$. We define: \begin{align*}
    \mathcal{M}_C\defeq\left\{A\in\mathcal{M}_d(\mathbb{R}): \forall i,j\in\llbracket 1,d\rrbracket, a_{i,j}>0 \text{ and } \frac{\max_{i,j}a_{i,j}}{\min_{i,j}a_{i,j}}\leq C\right\}.
\end{align*}
\end{definition}

\begin{definition}
Let $(A_n)_{n\in\mathbb{N}}\in\mathcal{M}_d(\mathbb{R})$. We define:\begin{align*}
    A^{(n)}=A_{n-1}\ldots A_0.
\end{align*}
\end{definition}

\begin{lemma}\cite[Lemma~2]{MR0121828}\label{lemmaproduct}
Let $C\geq 1$ and $(A_n)_{n\in\mathbb{N}}\in\mathcal{M}_C^\mathbb{N}$. Then, for all $n\in\mathbb{N}^*$, the matrix product $A^{(n)}\in\mathcal{M}_{C^2}$.
\end{lemma}

\begin{lemma}\label{lemmain}
Let $C\geq 1$ and $(A_n)_{n\in\mathbb{N}}\in\mathcal{M}_C^\mathbb{N}$. Then, for all $n\in\mathbb{N}^*$, and $(i,j)\in\llbracket 1,d\rrbracket^2$, \begin{align*}
    \liminf_{n\to +\infty}\frac{1}{n}\log (\|A^{(n)}\|)=\liminf_{n\to +\infty}\frac{1}{n}\log(\rho(A^{(n)}))=\liminf_{n\to +\infty}\frac{1}{n}\log ((A^{(n)})_{i,j}).
\end{align*}
\end{lemma}

\begin{proof} Let $n\in\mathbb{N}^*$. By Lemma~\ref{lemmaproduct}, $A^{(n)}\in\mathcal{M}_{C^2}$. 
By Perron–Frobenius theorem,
\begin{align*}
    \min_{i,j}(A^{(n)})_{i,j}&\leq \rho(A^{(n)}) \leq \|A^{(n)}\|\leq d\max_{i,j}(A^{(n)})_{i,j}\leq dC^2\min_{i,j}(A^{(n)})_{i,j}.\qedhere
\end{align*}
\end{proof}

While the norm $\|\cdot\|$ is naturally sub-multiplicative, the following lemma provides a reverse inequality, allowing us to bound the product from below.

\begin{lemma}\label{lemmainf}
Let $C\geq 1$ and $A,B\in \mathcal{M}_C$. Then,\begin{align*}
    \|AB\|\geq \frac{1}{C}\|A\|\|B\|.
\end{align*}
\end{lemma}

\begin{proof}
\begin{align*}
   \|AB\|&=\max_{i}\sum_{j=1}^d \sum_{l=1}^d a_{i,l}b_{l,j}\\
   &= \max_{i} \sum_{l=1}^d a_{i,l}\sum_{j=1}^d b_{l,j}\\
   &\geq \max_{i} \sum_{l=1}^d a_{i,l} \min_{s} \sum_{j=1}^d b_{s,j}\\
   &\geq \max_{i} \sum_{l=1}^d a_{i,l} \frac{1}{C}\max_{s} \sum_{j=1}^d b_{s,j}\\
   &\geq \frac{1}{C}\|A\|\|B\|.\qedhere
\end{align*}
\end{proof}

\paragraph{Exponential growth of the population}~\newline

In this paragraph, we assume Hypothesis~(H\ref{hmult2}) and that $\pi_\nu>0$ for all $\nu\in\mathcal{E}_T(\mathbb{X})$. Let $C=\frac{C_{\max}}{C_{\min}}\geq 1$.

By Hypothesis~(H\ref{hmult2})\ref{Hc}, for all $x\in\mathbb{X}$, the mean matrix $M(x)\in\mathcal{M}_C$. By Lemma~\ref{lemmaproduct}, for all $n\in\mathbb{N}^*$,\begin{align*}
    M^{(n)}(x)=M(T^{n-1}x)\ldots M(x)\in\mathcal{M}_{C^2},
\end{align*}
and $x\mapsto M^{(n)}(x)$ is continuous by Proposition~\ref{contM}.

For all $n\in\mathbb{N}^*$, let the continuous function $\psi_n$ be defined by:\fonction{\psi_n}{\mathbb{X}}{\mathbb{R}}{x}{-\log (\|M^{(n)}(x)\|)+\log(C^2)}

\begin{lemma}
The sequence of functions $(\psi_n)_{n\in\mathbb{N}^*}$ is sub additive, that is, for all $x\in\mathbb{R}$ and $n,m\in\mathbb{N}^*$, $\psi_{n+m}(x)\leq \psi_n(x)+\psi_m(T^nx)$.
\end{lemma}

\begin{proof}
Let $x\in\mathbb{X}$ and $n,m\in\mathbb{N}^*$. As, $M^{(m)}(T^nx),M^{(n)}(x)\in\mathcal{M}_{C^2}$, by Lemma~\ref{lemmainf}, \begin{align*}
 \|M^{(n+m)}(x)\|\geq \frac{1}{C^2}\|M^{(m)}(T^nx)\|\|M^{(n)}(x)\|.
\end{align*}
Thus, \begin{align*}
    \psi_{n+m}(x)&=-\log (\|M^{(n+m)}(x)\|)+\log(C^2)\\
    &\leq +\log(C^2)-\log( \|M^{(m)}(T^nx)\|)-\log (\|M^{(n)}(x)\|)+\log(C^2)\\
    &\leq \psi_m(T^nx)+\psi_n(x).\qedhere
\end{align*}
\end{proof}

\begin{lemma}\label{lemmaborne}
There exists $\delta>0$ and $N\in\mathbb{N}$ such that for all $n\geq N$ and for all $x\in\mathbb{X}$, \begin{align*}
    \frac{1}{n}\log (\|M^{(n)}(x)\|)\geq \delta.
\end{align*}
\end{lemma}

\begin{proof}
Let $\nu\in\mathcal{E}_T(\mathbb{X})$. For $\nu$ almost all $x\in\mathbb{X}$,\begin{align*}
    \lim_{n\to\infty}\frac{1}{n}\psi_n(x)=\lim_{n\to\infty}\frac{1}{n}\left(-\log (\|M^{(n)}(x)\|)+\log(C^2)\right)=-\pi_\nu<0.
\end{align*}
By \cite[Corollary~1.11]{MR1734626}, there exists $\delta>0$ and $N\in\mathbb{N}$ such that for all $n\geq N$ and $x\in\mathbb{X}$,\begin{align*}
    \frac{1}{n}\psi_n(x)=\frac{1}{n}\left(-\log (\|M^{(n)}(x)\|)+\log(C^2)\right)\leq -\delta.
\end{align*}
That is, for all $n\geq N$ and $x\in\mathbb{X}$,\begin{align*}
    \frac{1}{n}\log (\|M^{(n)}(x)\|)&\geq \delta+\frac{1}{n}\log(C^2)\geq \delta.\qedhere
\end{align*}
\end{proof}

\begin{corollary}\label{corobornemoy}
There exists $\delta>0$ and $N\in\mathbb{N}$ such that for all $n\geq N$, $x\in\mathbb{X}$, and $i,j\in\llbracket 1,d\rrbracket$, \begin{align*}
    (M^{(n)}(x))_{i,j}\geq e^{n\delta}.
\end{align*}
\end{corollary}

\begin{proof}
It is a direct corollary of Lemmas~\ref{lemmain} and~\ref{lemmaborne}.
\end{proof}

\paragraph{Upper semicontinuity}~\newline

\begin{lemma}\label{lemmadecr}
Assume Hypothesis~(H\ref{hmult2}) and that $\pi_\nu>0$ for all $\nu\in\mathcal{E}_T(\mathbb{X})$. Then, there exists $\bm{K}<\bm{1}$ and $N\in\mathbb{N}$ such that for all $x\in\mathbb{X}$, $\bm{\varphi}^{(N)}(x,\bm{K})\leq \bm{K}$.
\end{lemma}

\begin{proof}
By Corollary~\ref{corobornemoy}, let $\delta>0$ and $N\in\mathbb{N}$ be such that for all $n\geq N$, $x\in\mathbb{X}$, and $i,j\in\llbracket 1,d\rrbracket$, \begin{align*}
    (M^{(n)}(x))_{i,j}\geq e^{n\delta}.
\end{align*}
By Taylor-Young formula, for all $\varepsilon>0$, we have:
\begin{align*}
    \bm{\varphi}^{(N)}(x, \bm{1} - \epsilon \bm{1}) = \bm{\varphi}^{(N)}(x, \bm{1}) - \varepsilon M^{(N)}(x) \bm{1} + h(\epsilon,x).
\end{align*}
with $\frac{h(\epsilon,x)}{\varepsilon}\underset{\varepsilon\to 0}{\longrightarrow} \bm{0}$ and this convergence is uniform in the variable $x$ by (H\ref{hmult2})\ref{Hc}.
Moreover, because $\bm{\varphi}^{(N)}(x,\bm{1})=\bm{1}$ and by Corollary~\ref{corobornemoy}, we obtain that for $\varepsilon>0$ small enough:\begin{align*}
    \bm{\varphi}^{(N)}(x, \bm{1} - \epsilon \bm{1}) - (\bm{1}-\varepsilon\bm{1})=\varepsilon(\bm{1}-M^{(N)}(x)\bm{1}) + h(\epsilon,x)\leq \bm{0}.
\end{align*}
The conclusion follows for $\bm{K}=\bm{1}-\varepsilon\bm{1}$ with $\varepsilon>0$ such that the preceding inequality holds.
\end{proof}

\begin{lemma}\label{lemmalast}
Assume Hypothesis~(H\ref{hmult2}) and that $\pi_\nu>0$ for all $\nu\in\mathcal{E}_T(\mathbb{X})$. There exist $\bm{K}<\bm{1}$ and $N\in\mathbb{N}$ such that for all $x\in\mathbb{X}$, the sequence $(\bm{\varphi}^{(nN)}(x,\bm{K}))_{n\in\mathbb{N}}$ is decreasing and converges to $q(x)$.
\end{lemma}

\begin{proof}
Consider $\bm{K}<\bm{1}$ and $N\in\mathbb{N}$ as in Lemma~\ref{lemmadecr}, i.e., such that for all $x\in\mathbb{X}$, $\bm{\varphi}^{(N)}(x,\bm{K})\leq\bm{K}$. For all $n\in\mathbb{N}$,
\begin{align*}
    \bm{\varphi}^{((n+1)N)}(x,\bm{K})&=\bm{\varphi}^{(Nn)}(x,\bm{\varphi}^{(N)}(T^{nN}x,\bm{K})) \text{ by Proposition~\ref{mulprop5}}\\
    &\leq \bm{\varphi}^{(Nn)}(x,\bm{K}),
 \end{align*}
because $\bm{s}\in[0,1]^d\mapsto\bm{\varphi}^{(nN)}(x,\bm{s})$ is non-decreasing and $\bm{\varphi}(T^{Nn}x,\bm{K})\leq \bm{K}$.
Thus, the sequence $(\bm{\varphi}^{(nN)}(x,\bm{K}))_{n\in\mathbb{N}}$ is non-increasing. Moreover, for all $n\in\mathbb{N}$, $\bm{\varphi}^{(nN)}(x,0)\leq\bm{\varphi}^{(nN)}(x,\bm{K})$, so $\bm{q}(x)\leq \bm{K}$.

For $n\in\mathbb{N}$, $\bm{\varphi}^{(nN)}(x,\bm{0})\leq\bm{\varphi}^{(Nn)}(x,\bm{K})$, because $\bm{s}\in[0,1]^d\mapsto\bm{\varphi}^{(nN)}(x,\bm{s})$ is non-decreasing. Since $\bm{\varphi}^{(nN)}(x,\bm{0})$ converges to $\bm{q}(x)$, we only need to show that $\bm{\varphi}^{(nN)}(x,\bm{K})-\bm{\varphi}^{(nN)}(x,\bm{0})\underset{n\to +\infty}{\longrightarrow}\bm{0}$.

Kaplan~\cite{MR331477} study the composition limits of probability generating functions. By \cite[Theorems 1 and 2]{MR331477}, under (H\ref{hmult2})\ref{Hb} and \ref{Hc}, for all $x\in\mathbb{X}$, $\bm{\varphi}^{(Nn)}(x,\bm{K})$ converges. Moreover, for all $x\in\mathbb{X}$, \begin{align*}
    0 \leq \mu^{(i)}_{x}(\bm{0}) + \sum\limits_{j=1}^d \mu^{(i)}_{x}(\bm{e}_j) < 1 \text{ for all }x\in\mathbb{X},
\end{align*}and the upper bound is uniform in $x\in\mathbb{X}$ because $\mu$ is continuous. Thus, by \cite[Theorem 3]{MR331477}, the limit of $\bm{\varphi}^{(Nn)}(x,\bm{K})$ is equal to $\lim\limits_{n\to +\infty}\bm{\varphi}^{(nN)}(x,\bm{0})=\bm{q}(x)$.
\end{proof}

\begin{proof}[Proof of Theorem~\ref{thmcontinuite}]
The function $q$ is both a limit of a sequence of continuously non-decreasing functions $(x\mapsto\bm{\varphi}^{(n)}(x,\bm{0}))_{n\in\mathbb{N}}$ (by definition of $q$) and a limit of a sequence of continuously non-increasing functions $(x\mapsto\bm{\varphi}^{(nN)}(x,\bm{K}))_{n\in\mathbb{N}}$ (by Lemma~\ref{lemmalast}). Thus, each component of $\bm{q}$, is continuous because upper and lower semicontinuous, and so $\bm{q}$ is continuous.
\end{proof}

\begin{proof}[Proof of Corollary~\ref{multcor6}]
We define a function $\Phi$ as:\begin{align}\label{defphi}
\Phi : \left\{\begin{array}{lcl}   \mathcal{C}(\mathbb{X},[0,1]^d) &\to      &\mathcal{C}(\mathbb{X},[0,1]^d)\\
f    &\mapsto  &\left(x\in\mathbb{X}\mapsto\bm{\varphi}(x,f(Tx))\right)
\end{array} \right. .      
\end{align} 
 $\Phi$ is uniformly continuous. Indeed, let $\varepsilon>0$. The function $(x,s)\in\mathbb{X}\times[0,1]^d\mapsto\bm{\varphi}(x,s)$ is uniformly continuous by Corollary~\ref{mulcor5}. Hence, there exists $\eta>0$ such that for all $x\in\mathbb{X}$ and $s,t\in[0,1]^d$ with $\|s-t\|\leq \eta$, we have $\|\bm{\varphi}(x,s)-\bm{\varphi}(x,t)\|\leq\varepsilon$. Let $f,g\in\mathcal{C}(\mathbb{X},[0,1]^d)$ be such that $\lVert f-g \rVert_{\infty}\leq \eta$. Then, for all $x\in\mathbb{X}$, $
\|f(Tx)-g(Tx)\|\leq \eta$. 
Thus, for all $x\in\mathbb{X}$, $\|\bm{\varphi}(x,f(Tx))-\bm{\varphi}(x,g(Tx))\|\leq \varepsilon$, i.e.\ $\lVert \Phi(f)-\Phi(g) \rVert_{\infty}\leq \varepsilon$. Therefore, $\Phi$ is uniformly continuous.

Let $\bm{a}\in [0,1)^d$ and $\bm{K}<\bm{1}$ be defined as in Lemma~\ref{lemmadecr}. Let also $\widetilde{\bm{K}}=(\widetilde{K},\ldots,\widetilde{K})\in[0,1)^d$ be such that $\sup(\sup(\bm{a}),K)<\widetilde{K}$. The sequence of continuous and increasing functions $(\bm{\varphi}^{(n)}(\cdot,\bm{0}))_{n\in\mathbb{N}}$ (respectively the sequence of continuous and decreasing functions $(\bm{\varphi}^{(nN)}(\cdot,\widetilde{\bm{K}}))_{n\in\mathbb{N}}$) converges pointwise to the continuous (by Theorem~\ref{thmcontinuite}) function $q$. Moreover, since $\mathbb{X}$ is compact, by Dini's theorem (applied on each component of the vector), the sequences $(\bm{\varphi}^{(n)}(\cdot,\bm{0}))_{n\in\mathbb{N}}$ and $(\bm{\varphi}^{(nN)}(\cdot,\widetilde{\bm{K}}))_{n\in\mathbb{N}}$ converge uniformly to $\bm{q}$.
The uniform convergence of $(\bm{\varphi}^{(nN)}(\cdot,\widetilde{\bm{K}}))_{n\in\mathbb{N}}$ to $q$ can be expressed as:
\begin{align*}  
    \Phi^{nN}(\widetilde{\mathbb{K}})\underset{n\to +\infty}{\longrightarrow} q
\end{align*}
(where $\widetilde{\mathbb{K}}$ is the constant function equal to $\widetilde{\bm{K}}$).
For every $k\in\llbracket0,N-1\rrbracket$, by the continuity of $\Phi$ and by Proposition~\ref{mulprop1}:
\begin{align*}  
    \Phi^{nN+k}(\widetilde{\mathbb{K}})\underset{n\to +\infty}{\longrightarrow} \Phi^k(q)=q .
\end{align*}
Thus, $(\bm{\varphi}^{(n)}(\cdot,\widetilde{\bm{K}}))_{n\in\mathbb{N}}$ converges uniformly to $q$.
By squeezing, $(\bm{\varphi}^{(n)}(\cdot,\bm{a}))_{n\in\mathbb{N}}$ converges uniformly to $\bm{q}$.
\end{proof}

\section*{Acknowledgments}
The author thanks his thesis supervisor, Damien Thomine, for his advice and proofreading.

\bibliographystyle{alpha}
\bibliography{biblio.bib}

\end{document}